\documentclass[11pt]{amsart}

\usepackage{tikz}

\usetikzlibrary{knots}
\usetikzlibrary{decorations.markings}

\usepackage{graphicx} 
\usepackage{t1enc}
\usepackage{amsthm,amssymb}
\usepackage{latexsym}
\usepackage{amsmath} 
\usepackage{graphics}
\usepackage{epsfig,multicol}
\usepackage{amsfonts}
\usepackage[latin1]{inputenc}
\usepackage[english]{babel}
\usepackage{ mathrsfs }

\newtheorem{theorem}{Theorem}[section]
\newtheorem*{theorem*}{Theorem}
\newtheorem*{proposition*}{Proposition}
\newtheorem{proposition}{Proposition}[section]
\newtheorem{lemma}{Lemma}[section]
\newtheorem*{lemma*}{Lemma}

\newtheorem{remark}{Remark}[section]

\def\hpic #1 #2 {\mbox{$\begin{array}[c]{l} \epsfig{file=#1,height=#2}
\end{array}$}}
 
\def\vpic #1 #2 {\mbox{$\begin{array}[c]{l} \epsfig{file=#1,width=#2}
\end{array}$}}

\newcommand  {\rmn}\romannumeral

\newcommand{\IC}[0]{\mathbb{C}} 

 \newcommand{\IT}[0]{\mathbb{T}}

 \newcommand{\IZ}[0]{\mathbb{Z}}

 \newcommand{\CF}[0]{\mathcal{F}}

\newcommand{\CO}[0]{\mathcal{O}} 
\newcommand{\CQ}[0]{\mathcal{Q}} 
 
\newcommand{\CU}[0]{\mathcal{U}}

\newcommand{\TT}[0]{\tilde{T}}
\newcommand{\ts}[0]{\tilde{s}}

\begin{document}
\title[On the cyclic automorphism and its fixed-point algebra]{On the cyclic automorphism of the Cuntz algebra and its fixed-point algebra}
\author{Valeriano Aiello} 
\address{Valeriano Aiello,
Mathematisches Institut, Universit\"at Bern, Alpeneggstrasse 22, 3012 Bern, Switzerland
}\email{valerianoaiello@gmail.com}
\author{Stefano Rossi} 
\address{Stefano Rossi,
Dipartimento di Matematica, Universit\`a degli studi Aldo Moro di Bari, Via E. Orabona 4, 70125 Bari, Italy
}\email{stefano.rossi@uniba.it}

\begin{abstract}
We investigate the structure of the fixed-point algebra of $\CO_n$ under the action of  the cyclic permutation of the generating isometries.
We prove that it  is $*$-isomorphic with $\CO_n$, thus generalizing a result of Choi and Latr\'emoli\`ere on $\CO_2$. 
As an application of the technique employed, we also describe the fixed-point algebra of $\CO_{2n}$ under the exchange automorphism.
\end{abstract}

\maketitle

\section{Introduction}
In \cite{CL} Choi and  Latr\'emoli\`ere proved that the so-called flip-flop automorphism $\lambda_f$ of the Cuntz algebra $\CO_2$ enjoys rather remarkable properties.
We recall that $\CO_2$ is the universal $C^*$-algebra generated by two isometries, $S_1$ and $S_2$, such that $S_1S_1^*+ S_2S_2^*=1$, and
$\lambda_f \in{\rm Aut}(\CO_2)$ acts by switching $S_1$ and $S_2$ with each other. By definition the flip-flop is an automorphism of order $2$. In particular, it implements
an automorphic action of $\IZ_2$ on $\CO_2$. Now actions of as simple mathematical objects as finite groups on $C^*$-algebras may nevertheless  give rise to fixed-point subalgebras or crossed products
which are difficult to describe concretely, cf. \cite{Izumi1, Izumi2}. In fact, the flip-flop is one of the few known examples of any interest where the associated structures are not only comparatively
easy to describe but  also very surprising. Indeed, both fixed-point subalgebra and crossed product by its action
turn out to be $*$-isomorphic with $\CO_2$ itself, \cite{CL}. To the best of our knowledge, not many examples of this sort have appeared in the literature.
Therefore, the present note aims to exhibit more examples with properties similar to those of the flip-flop. It is quite natural to try
to find such examples on the Cuntz algebras $\CO_n$ for any natural $n\geq 2$, and this is exactly what is done here.
Now there are at least two ways to define a generalized flip-flop on $\CO_n$.
The first we consider is what we  call the cyclic automorphism $\lambda_C$, which act on the generating isometries by translation mod $n$.
This is now an automorphism of order $n$. Interestingly, its fixed-point subalgebra is still $*$-isomorphic with $\CO_n$, as we show
in Section \ref{cyclic}. One may of course ask whether $\CO_n$ is acted upon by an automorphism of order $2$ such that
the corresponding fixed-point subalgebra is again $*$-isomorphic with $\CO_n$. It is this question that naturally leads us to consider
the so-called exchange automorphism, already appeared in \cite{ACRS}. This is the automorphism $\lambda_E$ associated with the $n\times n$ matrix $E=(\delta_{n-i+1,j})_{i,j=1}^n$. 
Unlike the cyclic automorphism, the exchange automorphism can fix one of the generating isometries. This happens 
precisely when $n$ is odd, and in this case the fixed isometry is $S_{\frac{n+1}{2}}$. The exchange automorphism, though, proves to be far more difficult to treat.
We describe the fixed-point subalgebra of the exchange subalgebra only when it  does not fix any of the generating isometries, namely when $n$ is even.
The resulting $C^*$-algebra is no longer a Cuntz algebra. However, it is generated by $n$ copies of $\CO_n$, as we show in Section
\ref{exchange}.

Quite interestingly, the main ingredient of the result of Choi and Latr\'emoli\`ere is a convenient realization
of $\CO_2$ as a subalgebra of the tensor product  $M_2(\IC)\otimes\CO_2\cong M_2(\CO_2)$, in which the flip-flop, despite being outer, can be implemented by
a unitary $Z\in M_2(\CO_2)$ of a particularly simple form. However, this realization fails to be compatible
with the inclusion of $\CO_2$ in the so-called $2$-adic ring $C^*$-algebra $\CQ_2$, which has been the focus of much recent work \cite{LarsenLi, ACR, ACR2, ACR3, ACR4, ACR5, AR, ACRsurvey}, in a sense that we make
precise in Section \ref{nogo}.

\section{Preliminaries and Notation}
We denote by $\CO_n$   the universal C$^*$-algebra generated by $n$ (proper) isometries $S_1, \ldots , S_n$ such that $\sum_{i=1}^nS_iS_i^*=1$. It is well known that
$\CO_n$ is a simple $C^*$-algebra, \cite{Cuntz1}. We denote by $\CF_n^k\subset \CO_n$ the linear span of words in the generating isometries of the type
$S_{\alpha_1}S_{\alpha_2}\cdots S_{\alpha_k}S_{\beta_1}^*S_{\beta_2}^*\cdots S_{\beta_k}^*$ with $\alpha_i, \beta_i\in\{1, 2, \ldots, n\}$
for any $i=1, 2, \ldots, k$. 
Endomorphims of $\CO_n$ are in bijection with its unitaries via the so-called Cuntz-Takesaki correspondence, cf. \cite{Cuntz2}.
This is realized as follows. To any unitary $u\in\CO_n$ it is possible to associate an endomorphims $\lambda_u$ defined by
$\lambda_u(S_i):=uS_i$, $i=1, 2, \ldots, n$.  Conversely, if one starts with an endomorphism $\lambda$, it is easy to see that
$u:=\sum_{i=1}^n\lambda(S_i)S_i^*$ is a unitary and $\lambda=\lambda_u$.\\
On the Cuntz algebra $\CO_2$ we can consider the so-called
 flip-flop automorphism $\lambda_f\in $Aut$(\CO_2)$, first appeared under this name in \cite{Archbold}. This is nothing but the involutive automorphism that switches the generating isometries $S_1$ and $S_2$, namely
$\lambda_f(S_1)=S_2$, $\lambda_f(S_2)=S_1$. Obviously, $\lambda_f$ is associated with the unitary $S_1S_2^*+S_2S_1^*$.
The cyclic automorphism $\lambda_C\in{\rm End}(\CO_n)$ is defined as $\lambda_C(S_i)=S_{i+1}$ for $i=1, \ldots , n-1$ and   $\lambda_C(S_n)=S_{1}$. If sums between indices are understood mod $n$, then the foregoing formula
simply becomes $\lambda_C(S_i)=S_{i+1}$ for every $i=1, 2, \ldots, n$.
It is easy to see that $\lambda_C$ is associated with the unitary matrix $C=(\delta_{i-j,1})_{i,j=1}^n$  identified with the corresponding  unitary in $\CF_n^1$ (again, the operations between indices are understood mod $n$). Note that when $n=2$, $\lambda_f=\lambda_C$.
The exchange automorphism  (introduced in \cite{ACRS}) is the automorphism $\lambda_E$, with $E=(\delta_{n-i+1,j})_{i,j=1}^n\in \CF_n^1$. 
Again, when $n=2$ we have $\lambda_E=\lambda_f$.
This means that both exchange automorphism and cyclic automorphism may be regarded as natural generalizations of the flip-flop.
Finally, the $2$-adic ring $C^*$-algebra $\CQ_2$ is the universal C$^*$-algebra generated by a unitary $U$ and a (proper) isometry $S_2$ such that $S_2U=U^2S_2$ and $S_2S_2^*+US_2S_2^*U^*=1$. 
It contains a copy of the Cuntz algebra $\CO_2$ since $S_2$ and $S_1:=US_2$ satisfy the Cuntz relations. 
As shown in \cite{ACR}, the flip-flop uniquely extends  to an automorphism $\widetilde{\lambda_f}$ of $\CQ_2$, determined $\lambda_f(U):=U^*$.

Throughout the paper sums and differences between indices of the generating isometries of $\CO_n$  will always be understood mod $n$.

\section{The fixed-point algebra under the cyclic automorphism}\label{cyclic}
In \cite[Section 1]{CL}, Choi and Latr\'emoli\`ere proved that the fixed-point algebra of $\CO_2$ with respect to the action of the flip-flop automorphism is $*$-isomorphic with $\CO_2$.
This section shows how their result can be generalized to $\CO_n$, for every $n\geq 2$, when the flip-flop is replaced by the cyclic automorphism $\lambda_C$. We first discuss
  the spectral  decomposition of $\CO_n$ into the direct sum of the $n$ eigenspaces of $\lambda_C$, which are
commonly referred to as the
spectral eigenspaces associated with the action of $\IZ_n$ through $\lambda_C$. Henceforth $i$ will always denote
the imaginary unit of $\IC$ when it appears in the argument of the exponential function.

\begin{proposition}\label{prop31}
The algebra $\CO_n$ decomposes as a direct sum of Banach spaces  $\oplus_{k=1}^n \CO_n^{\lambda_C}v^k$, where $v=\sum_{k=1}^n e^{2\pi k i/n}	S_kS_k^*\in\CU(\CO_n)$, and $\CO_n^{\lambda_C}$ is the fixed-point  subalgebra of $\CO_n$ w.r.t. the cyclic automorphism.
\end{proposition}
\begin{proof}
First we check that $v$ is an eigenvector corresponding to the eigenvalue $e^{-2\pi i/n}$. Indeed,
\begin{align*}
e^{2\pi i/n}\lambda_C(v)&=\left(\sum_{k=1}^{n-1} e^{2\pi k i/n}	S_{k+1}S_{k+1}^*+S_1S_1^*\right)\\
&=e^{2\pi i/n}\left(\sum_{h=2}^{n} e^{2\pi (h-1) i/n}	S_{h}S_{h}^*+S_1S_1^*\right)\\
&=e^{2\pi i/n}\left(\sum_{h=1}^n e^{2\pi h i/n}	S_{h}S_{h}^*\right)\\
&=v\; .
\end{align*}
Accordingly, $v^j$ is an eigenvector  corresponding to the eigenvalue $e^{-2\pi ji/n}$, for any $j=1, 2, \ldots, n$.
The conditional expectation onto $\CO_n^{\lambda_C}$ is now obtained by averaging the action
of $\IZ_n$ through $\lambda_C$, that is
$$
F(x):=\frac{1}{n}\sum_{k=1}^n\lambda_C^{k-1}(x)\qquad x\in\CO_n\; .
$$
\noindent
For every $x\in \CO_n$, we note that $x=\sum_{k=0}^{n-1}F(xv^k)v^{-k}$; the terms of the sum on the right-hand side of this equality are called the spectral components of $x$ associated, respectively, with  the  eigenvalue
$e^{\frac{2\pi ki}{n}}$.

\end{proof}
The following technical lemma is at the core of the proof of the main result of the present section. First, it  provides an explicit  copy
of $\CO_n$ as a subalgebra of $M_n(\CO_n)$ by exhibiting the new generators. Second, it shows  that  the copy of $\CO_n$ thus obtained is a small subalgebra
of $M_n(\CO_n)$ insofar as
the entries of a $n\times n$ matrix of $M_n(\CO_n)$ sitting in this copy are constrained by rigid relations involving the cyclic automorphism itself.
\begin{lemma}\label{lemmaFix}
Let $V:=(V_{i,j})_{i,j=1}^n$ be the matrix in $M_n(\CO_n)$  with $V_{i, j}=S_j$ for any $i,j=1, 2, \ldots , n$.
Let $Z$ be the diagonal matrix $(e^{2\pi (k-1)i/n}\delta_{k,h})_{h,k=1}^n$. Then, the C$^*$-algebra generated by 
$
T_l:=\frac{1}{\sqrt{n}}Z^{l-1}VZ^{-l+1}
$, $l=1, 2, \ldots , n$, 
is $*$-isomorphic with $\CO_n$. Moreover, $C^*(T_1, T_2, \ldots, T_n )$ consists of all the elements 
$A=(A_{h,k})_{h,k=1}^n$ in $M_n(\CO_n)$ such that $A_{h+1,k+1}=\lambda_C(A_{h,k})$ for all
$h, k=1, 2, \ldots ,  n$ (mod $n$).
\end{lemma}
\begin{proof}
It is a matter of routine computations to see that 
$$
(T_l)_{h, k}=\frac{1}{\sqrt{n}} S_k e^{\frac{2\pi i l(h-k)}{n}}
$$ 
for any
$l, h, k=1, 2, \ldots , n$.  We first show that each $T_l$ is an isometry. To this end, it is enough to make the relative
computation only for $T_1$. For any $h, k=1, 2, \ldots, n$, we have
\begin{align*}
(T_1^*T_1)_{h,k}&=\sum_{j=1}^n (T_1^*)_{h,j}(T_1)_{j,k}=\sum_{j=1}^n (T_1)_{j,h}^*(T_1)_{j,k}\\
 &=\sum_{j=1}^n S_{h}^*S_{k}n^{-1}
 =\delta_{h, k}
\end{align*}
In order to show that $C^*(T_1, T_2, \ldots, T_n)$ is $^*$-isomorphic
with $\CO_n$, we also need show that $\sum_{l=1}^n T_lT_l^*=1$. For any $h, k=1, 2, \ldots, n$ we have
\begin{align*}
\sum_{l=1}^n(T_lT_l^*)_{h,k}&= \sum_{l=1}^n\sum_{j=1}^n (T_l)_{h,j} (T_l^*)_{j,k}\\
&= \sum_{l=1}^n\sum_{j=1}^n (T_l)_{h,j} (T_l)_{k,j}^*\\
&=n^{-1}\sum_{l=1}^n\sum_{j=1}^n (S_{j}e^{2\pi i (h-j)l/n})(S_{j}e^{2\pi i (k-j)l/n})^*\\
&=n^{-1}\sum_{l=1}^n\sum_{j=1}^n (S_{j}e^{2\pi i (h-j)l/n})(S_{j}^*e^{2\pi i (-k+j)l/n})\\
&=n^{-1}\sum_{l=1}^n\sum_{j=1}^n S_{j}S_{j}^*e^{2\pi i (h-k)l/n}\\
&=n^{-1}\sum_{j=1}^n S_{j}S_{j}^*\left(\sum_{l=1}^n e^{2\pi i (h-k)l/n}\right)=\delta_{h,k}I
\end{align*}
and the claimed equality is thus proved.
This allows us to define an automorphism $\alpha: \CO_n\to C^*(T_1,\ldots, T_n)$ given by
\begin{equation}\label{alpha}
\alpha(S_k)=T_k, \quad k=1,2, \ldots, n.\\
\end{equation}

\smallskip

Let $A:=(a_{hk})_{h, k=1}^n$ and $B:=(b_{hk})_{h, k=1}^n$ two matrices such that $a_{h+1,k+1}=\lambda_C(a_{h,k})$ and $b_{h+1,k+1}=\lambda_C(b_{h,k})$. Then it is easy to see that $(c_{h,k})_{h,k=1}^n=AB$ enjoys the same property, namely $c_{h+1,k+1}=\lambda_C(c_{h,k})$ for any $h$, $k=1,\ldots , n$.
By definition $T_1, \ldots, T_n$ enjoy this property and thus all the elements $A=(A_{h,k})_{h,k=1}^n$ in  $C^*(T_1, T_2, \ldots, T_n )$ satisfy the condition $A_{h+1,k+1}=\lambda_C(A_{h,k})$ for all
$h, k=1, 2, \ldots ,  n$.\\

To prove the converse inclusion, we need to make some preliminary computations.\\
We  observe that $w:=\alpha(v)$ is equal to the matrix $(\delta_{p-q+1, 0})_{p,q=1}^n$, where $v=\sum_{k=1}^n e^{2\pi k i/n}	S_kS_k^*$ is 
the unitary considered in Proposition \ref{prop31}. Indeed, for 
any $p, q=1, 2, \ldots, n$ we have
\begin{align*}
w_{p,q}=\alpha(v)_{p,q}&=\sum_{k=1}^n e^{2\pi i k/n} (T_kT_k^*)_{p,q}\\
&=\sum_{k=1}^n\sum_{a=1}^n e^{2\pi i k/n} (T_k)_{p,a}(T_k^*)_{a,q}\\
&=\sum_{k=1}^n\sum_{a=1}^n e^{2\pi i k/n} (T_k)_{p,a}(T_k)_{q,a}^*\\
&=\sum_{k=1}^n\sum_{a=1}^n e^{2\pi i k/n} S_{a} S_{a}^*e^{2\pi i \frac{k(p-a+a-q)}{n}}\\
&=\sum_{k=1}^n\sum_{a=1}^n e^{2\pi i k/n} S_{a} S_{a}^*e^{2\pi i \frac{k(p-q)}{n}}\\
&=\sum_{a=1}^n \delta_{p-q+1,0} S_{a} S_{a}^* =\delta_{p-q+1, 0}I.
\end{align*}
We also observe that for $j=1, \ldots, n$, one has $(w^j)_{h,k}=\delta_{k-h,-j}$  for any $h$, $k=1,\ldots , n$. We give a proof by induction on $j$. The formula is true for $j=1$. Suppose that it is true for $j\geq 1$. 
Now for $j+1\leq n$ we have 
\begin{align*}
(w^{j+1})_{h,k}=(w^jw)_{h,k}&=\sum_{p=1}^n(w^j)_{h,p}(w)_{p,k}\\
&=\sum_{p=1}^n \delta_{p-h,-j}  \delta_{k-p,-1} =\delta_{k-h,-j-1}
 \end{align*}

Set now $T:=\frac{1}{\sqrt{n}}\sum_{k=1}^n S_k$ and $s_l:=(\lambda_C^{h-1}(S_l)\delta_{h,k})_{h,k=1}^n$, $l=1,\ldots , n$. We aim to verify that $s_l$ sits in $C^*(T_1, T_2, \ldots, T_n)$ for every $l=1, 2, \ldots, n$. To this end,
we first  show that $\alpha(T)=s_1$. Indeed, for any  $p, q=1, 2, \ldots, n$ we have
\begin{align*}
\alpha(T)_{p,q}&=\frac{1}{\sqrt{n}} \left(\sum_{l=1}^n (T_l)_{p,q}\right)\\
&=\frac{1}{\sqrt{n}} \left(\sum_{l=1}^n \frac{1}{\sqrt{n}} S_q e^{\frac{2\pi i l(p-q)}{n}}\right)\\
&=\frac{1}{n} S_q \left(\sum_{l=1}^n  e^{\frac{2\pi i l(p-q)}{n}}\right)\\
&=\left\{\begin{array}{cc}
S_p & \text{ if $p=q$}\\
0 & \text{ if $p\neq q$}
\end{array}\right.\\
&=(s_1)_{p,q}
 \end{align*}

In particular, $s_1$ is in $C^*(T_1, T_2, \ldots, T_n)$. We will see that the remaining $s_l$ are in 
$C^*(T_1, T_2, \ldots, T_n)$ by showing the equality $ws_lw^*=s_{l-1}$, which holds for any $l=1, 2, \ldots, n$ with the convention that for $l=1$ one has $ws_1w^*=s_{n}$. To this aim, we point out that 
 for any $h$, $k=1,\ldots , n$,  $(w^js_l)_{h,k}=S_{k+l-1}\delta_{k-h, -j}$, which is verified below
\begin{align*}
(w^js_l)_{h,k}&=\sum_{p=1}^n(w^j)_{h,p}(s_l)_{p,k}\\
&=\sum_{p=1}^n \delta_{p-h,-j} S_{l+k-1} \delta_{p,k}=S_{k+l-1}\delta_{k-h,-j}
 \end{align*}
But then we have
\begin{align*}
(ws_lw^*)_{h,k}&=\sum_{p=1}^n(ws_l)_{h,p}(w^*)_{p,k}\\
&=\sum_{p=1}^n   S_{p+l-1}\delta_{p-h,-1}    (w)_{k,p}\\
&=S_{h-1+l-1}\delta_{p-h,-1}    \delta_{p-k, -1}\\
&= S_{h+l-2} \delta_{h,k}=(s_{l-1})_{h,k}
 \end{align*}

 
\bigskip 

We finally move on to ascertain that the converse inclusion holds as well. First we observe that, for any $x\in \CO_n$, the diagonal matrix of the form $(\lambda_C^{h-1}(x)\delta_{h,k})_{h,k=1}^n$, $l=1,\ldots , n$, belongs to $C^*(T_1,\ldots , T_n)$. Indeed, this follows from the fact that $(\lambda_C^{h-1}(S_l)\delta_{h,k})_{h,k=1}^n$ is
equal to $s_l$ by definition and $s_l\in C^*(T_1,\ldots , T_n)$, for every $l=1,\ldots , n$.\\
Now let  $A=(A_{h,k})_{h,k=1}^n$ be a matrix in $M_n(\CO_n)$ such that $A_{h+1,k+1}=\lambda_C(A_{h,k})$ for all
$h, k=1, 2, \ldots ,  n$. In particular, $A$ is determined by the elements $A_{1,1}$, \ldots, $A_{1,n}\in \CO_n$.
For any $l=1,\ldots , n$, the diagonal matrices $D_l:=(\lambda_C^{h-1}(A_{1,l})\delta_{h,k})_{h,k=1}^n$ are in $C^*(T_1,\ldots , T_n)$.
Since $w=(\delta_{p-q+1, 0})_{p,q=1}^n, D_1, \ldots , D_n$ all belong to $C^*(T_1,\ldots , T_n)$, it suffices to show that $A=\sum_{l=0}^{n-1} w^lD_{l+1}$. But this can be ascertained by means of the following computations
\begin{align*}
\left( \sum_{l=0}^{n-1} w^l D_{l+1}\right)_{h,k}&=\sum_{p=1}^n(w^l)_{h,p}(D_{l+1})_{p,k}\\ 
&=\sum_{l=0}^{n-1}\sum_{p=1}^n  \delta_{p-h, -l}  \lambda_C^{k-1}(A_{1,l+1})\delta_{p,k} \\ 
&=    \lambda_C^{k-1}(A_{1,h-k+1})= A_{h,k}   
 \end{align*}
\end{proof}
As an easy yet useful application of the previous lemma, we get the following corollary, where we keep the same notation as in Lemma \ref{lemmaFix}.
\begin{lemma}\label{lemmaFix2}
The fixed-point algebra of $C^*(T_1,\ldots, T_n)$ under the 
order-$n$ automorphism Ad$(Z)$, with ${\rm Ad}(Z)(A):=ZAZ^*$ for $A\in C^*(T_1, T_2, \ldots, T_n)$, is given by the diagonal matrices $(\lambda_C^{h-1}(x)\delta_{h,k})_{h,k=1}^n\in M_n(\CO_n)$, where $x$ varies in $\CO_n$.
\end{lemma}
\begin{proof}
Let $A$ be in $M_n(\CO_n)$. It is easy to see that the entry $(h,k)$ of Ad$(Z)(A)$ is $A_{h,k}e^{2\pi i (h-k)/n}$ and, therefore, it is fixed only if $h=k$ or $A_{h,k}=0$.
\end{proof}
We are finally in a position to prove the announced result.
\begin{theorem}
The fixed-point algebra of $\CO_n$ under $\lambda_C$ is $^*$-isomorphic with $\CO_n$.
Moreover, $\CO_n^{\lambda_C}$ is generated by ${\rm Ad}(v^l)(T)$ with $l=0, 1, \ldots, n-1$, where $T=\frac{1}{\sqrt{n}}\sum_{k=1}^n S_k$ and $v=\sum_{k=1}^n e^{2\pi k i/n}S_kS_k^*$.
\end{theorem}
\begin{proof}
Recall that by \eqref{alpha} $\alpha: \CO_n\to C^*(T_1,\ldots, T_n)$ is the $*$-isomorphism mapping $S_k$ to $T_k$, $k=1, 2, \ldots n$.
By definition $\alpha\circ \lambda_C={\rm Ad}(Z)\circ\alpha$.
It follows from the previous lemma that the fixed-point subalgebra of $\CO_n$ under $\lambda_C$ is $*$-isomorphic with $\CO_n$.
In addition, this copy of $\CO_n$ in $C^*(T_1, T_2, \ldots, T_n)$
 is generated by $s_l:=(\lambda_C^{h-1}(S_l)\delta_{h,k})_{h,k=1}^n$, $l=1,\ldots , n$.
The generators of $\CO_n^{\lambda_C}\subset\CO_n$ can thus be obtained as $\alpha^{-1}(s_l)$.
Indeed, we saw in the proof of Lemma \ref{lemmaFix}
 that $\alpha(T)=s_1$ and 
$\alpha({\rm Ad}(v^l)(T))=s_{1-l+n}$ for all $l=0,\ldots , n-1$.

\end{proof}
\begin{remark}
The inclusion  $\CO_n^{\lambda_C}\subset \CO_n$ is an example of a non-commutative self-covering of the type considered in \cite{AGI}. 
\end{remark}

\section{An application: the fixed-point algebra of the exchange automorphism of $\CO_{2n}$}\label{exchange}
We want to exploit the analysis of the previous section to describe the subalgebra $\CO_{2n}^{\lambda_E}\subset\CO_{2n}$ of all elements
fixed by the exchange automorphism. Notice that we are only dealing with Cuntz algebras associated with even numbers of  generating isometries.
If $Z\in M_{2n}(\CO_{2n})$ is the matrix  $(e^{2\pi (k-1)i/2n}\delta_{k,h})_{h,k=1}^{2n}$, we can define an inner automorphism $\rho$
by setting $\rho:={\rm Ad}(Z)^n={\rm Ad}(Z^n)$. Obviously, $\rho$ is an automorphims $\CO_{2n}$ of order $2$.
To ease the computations that we will need to make, we prefer to rename the generating isometries considered in the previous section.
We set $\TT_{k}:=T_k$ for all $k=1, 	\ldots , n$,
$\TT_{2n-k+1}:=\rho(\TT_k)=\rho(T_k)=T_{k+n}$.
By Lemma \ref{lemmaFix} the $C^*$-algebra generated by $\TT_1, \ldots , \TT_{2n}$ is $*$-isomorphic with $\CO_{2n}$ and by construction the exchange automorphism on this algebra coincides with $\rho$.
We denote by $\beta$ the $*$-isomorphism mapping $S_k$ to $\TT_k$ for $k=1, 	\ldots , 2n$.
As before, unless otherwise stated, operations between indices are always meant mod $2n$.
\begin{lemma}\label{lemmaFix2Ex}
The fixed-point algebra of $C^*(\TT_1,\ldots, \TT_n)$ under the 
order-$2$ automorphism $\rho$ is given by the matrices $A$ in $M_n(\CO_n)$ whose entries are all zero except those with indices $(h,k)$ with $h-k\in 2\IZ$
and such that $\lambda_C(A_{h,k})=A_{h+1,k+1}$ for any $h, k=1, 	\ldots , n$.
\end{lemma}
\begin{proof}
Easy computations show that for any $h, k=1, 2, \ldots, n$ the entry   $\rho(A)_{h,k}$ is $(-1)^{h-k}A_{h,k}$.
The remaining part of the claim follows from Lemma \ref{lemmaFix2}.
\end{proof}
\begin{theorem}
The fixed-point algebra of $\CO_{2n}$ under $\lambda_E$ is generated by
$y^{j}_l:=(2n)^{-1/2}{\rm Ad}(\tilde v^l)(( \tilde v^{2j}T))$ for $j=0, \ldots , n-1$, $l=0, \ldots, 2n-1$, where 
\begin{align*}
&T=\frac{1}{\sqrt{2n}}\sum_{k=1}^{2n} S_k \\
&\tilde v=\sum_{k=1}^n e^{2\pi k i/(2n)}	S_kS_k^*+\sum_{k=1}^n e^{2\pi (k+n) i/(2n)}	S_{2n-k+1}S_{2n-k+1}^* 
\end{align*}
Moreover, for any fixed $j$, $C^*(\{y_l^j\}_{l=1}^{2n})$ is $*$-isomorphic with $\CO_{2n}$.
\end{theorem}
\begin{proof}
By Lemma \ref{lemmaFix2Ex}, an element in the fixed-point algebra is determined by the $n$ non-zero elements on the first row of the matrix. 
The elements $$(\ts_l^j)_{h,k}:=(2n)^{-1/2}{\rm Ad}(w^l)(( w^{2j}s_1))_{h,k}=(2n)^{-1/2}S_{k+l}\delta_{k-h, -2j}$$ 
$j=0, \ldots , n-1$ and $l=0,\ldots , 2n-1$ are easily seen to generate the algebra. 
For every fixed $j$, the $C^*$-algebra $C^*(\{\ts^j_l\}_{l=0}^{2n-1})$ is $*$-isomorphic with $\CO_{2n}$. In other terms, the isometries of the set $\{\ts^j_l\}_{l=0}^{2n-1}$ satisfy the Cuntz
relations, as easily follows from the equality $\tilde{s}_l^j=w^{2j}s_l$ for any $j=0, 1, \ldots n-1$ and $l=0, 1, \ldots, 2n-1$.\\
Now in order to find the generators of $\CO_{2n}^{\lambda_E}$, it suffices to compute $\beta^{-1}(\ts_0^0)$ and $\beta^{-1}(w)=\beta^{-1}(\alpha(v))$.
Clearly, $\beta^{-1}(\ts_0^0)=\alpha^{-1}(s_1)=T$. As for $\beta^{-1}(w)$, the following computations show it equals
the $\widetilde{v}$ in the statement. Indeed, one has
\begin{align*}
\beta(\tilde v)&=\beta\left( \sum_{k=1}^n e^{2\pi k i/(2n)}	S_kS_k^*+\sum_{k=1}^n e^{2\pi (k+n) i/(2n)}	S_{2n-k+1}S_{2n-k+1}^*\right)\\
&=\sum_{k=1}^n e^{2\pi k i/(2n)}	T_kT_k^*+\sum_{k=1}^n e^{2\pi (k+n) i/(2n)}	T_{k+n}T_{k+n}^*=\alpha(v)
\end{align*}
\end{proof}

\section{A no-go result on a Choi-Latr\'emoli\`ere type matrix model for $\CQ_2$}\label{nogo}
When $n=2$, the construction described in the previous section provides us with a pair of orthogonal  isometries $T_1$ and $T_2$ that generate a (proper) C$^*$-subalgebra of $M_2(\IC)\otimes \CO_2$
 $*$-isomorphic with $\CO_2$.  The generating isometries take now the simpler form
$$
T_1=\frac{1}{\sqrt{2}}\left(\begin{array}{cc}
S_1 & S_2\\
S_1 & S_2\\
\end{array}\right)
\qquad 
T_2=\frac{1}{\sqrt{2}}\left(\begin{array}{cc}
S_1 & -S_2\\
-S_1 & S_2\\
\end{array}\right)
$$
The principal merit of this construction of the Cuntz algebra $\CO_2$, which was first pointed out in \cite[Lemma 1.2]{CL}, is that  it is  in a sense the more economical 
realization in which the flip-flop is implemented by a unitary of a particularly simple form. Indeed,  the flip-flop $\lambda_f$ can be seen as the restriction to $\CO_2$
of the inner automorphism implemented by the unitary $Z$, \cite[Lemma 1.3]{CL}, given by
$$
Z=\left(\begin{array}{cc}
1 & 0\\
0 & -1\\
\end{array}\right)
$$
More explicitly, the equalities $ZT_2Z^*=T_1$ and $ZT_1Z^*=T_2$ are seen at once to hold.\\
It is worth recalling that the flip-flop is well known to be an outer automorphism, which makes the above presentation of
$\CO_2$ as a subalgebra of $M_2(\CO_2)$ even more interesting. This obviously implies that  $Z$ cannot belong to $C^*(T_1,T_2)\cong \CO_2$.
In fact, adding $Z$ to $C^*(T_1, T_2)$ yields the whole $M_2(\IC)\otimes \CO_2$, as observed in \cite[Remark 2.4]{CL}.
In \cite{ACR} the flip-flop was shown to uniquely extend to an automorphism $\widetilde{\lambda_f}$ of the $2$-adic ring $C^*$-algebra $\CQ_2$ with
$\widetilde{\lambda_f}(U)=U^*$.  Furthermore, its action on $\CQ_2$ is still outer. This is actually
a consequence of a general result proved in \cite{ACR} that any automorphism of $\CQ_2$ that sends $U$ to $U^*$ is automatically outer.
At this point one might wonder whether the realization of $\CO_2$ as a subalgebra of $M_2(\CO_2)$ is compatible with the inclusion
$\CO_2\subset\CQ_2$. Phrased differently, one may seek to find a unitary $V\in M_2(\CQ_2)\supset M_2(\CO_2)$ such that
$T_2V=V^2T_2$ and $T_1=VT_2$ in such a way that $\widetilde{\lambda_f}$ is still implemented by $Z$.
Possibly because of the rigidity of the inclusion of the Cuntz algebra in the $2$-adic ring $C^*$-algebra, which is discussed at length in \cite{ACR},
the answer to this question is negative, as the result below shows.

\begin{theorem}
There exists no   unitary $V\in M_2(\IC)\otimes \CQ_2$ such that $VT_2=T_1$, $VT_1=T_2V$, and
$ZVZ^*=V^*$.
\end{theorem}

The proof is by contradiction. Le us suppose that there does exist 
 a unitary  
$$
V=\left(\begin{array}{cc}
a & b\\
c & d\\
\end{array}\right)\in M_2(\IC)\otimes \CQ_2
$$
that enjoys the properties listed in the statement, for some $a, b, c, d \in\CQ_2$.
The condition  $VT_2=T_1$ reads as 
\begin{align*}
VT_2&=\left(\begin{array}{cc}
a & b\\
c & d\\
\end{array}\right)
\frac{1}{\sqrt{2}}\left(\begin{array}{cc}
S_1 & -S_2\\
-S_1 & S_2\\
\end{array}\right)=
\frac{1}{\sqrt{2}}\left(\begin{array}{cc}
S_1 & S_2\\
S_1 & S_2\\
\end{array}\right)=T_1
\end{align*}
while $T_2V=VT_1$  rewrites as
\begin{align*}
VT_1&=\left(\begin{array}{cc}
a & b\\
c & d\\
\end{array}\right)
\frac{1}{\sqrt{2}}\left(\begin{array}{cc}
S_1 & S_2\\
S_1 & S_2\\
\end{array}\right)=
\frac{1}{\sqrt{2}}\left(\begin{array}{cc}
S_1 & -S_2\\
-S_1 & S_2\\
\end{array}\right)
\left(\begin{array}{cc}
a & b\\
c & d\\
\end{array}\right)
=T_2V
\end{align*}
Reading the former equalities componentwise, we are led to the following system of equations
\begin{align}
&aS_1-bS_1=S_1\\
&-aS_2+bS_2=S_2\\
&cS_1-dS_1=S_1\\
&-cS_2+dS_2=S_2\\
&aS_1+bS_1=S_1a-S_2c\\
&aS_2+bS_2=S_1b-S_2d\\
&cS_1+dS_1=-S_1a+S_2c\\
&cS_2+dS_2=-S_1b+S_2d
\end{align}
Equations (1) and (2) yield
$$
aS_1S_1^*-bS_1S_1^*+aS_2S_2^*-bS_2S_2^*=a-b=S_1S_1^*-S_2S_2^*=:F
$$
We observe that $F$ is a self-adjoint unitary of $\CQ_2$. Similarly (3) and (4) give
$$
cS_1S_1^*-dS_1S_1^*+cS_2S_2^*-dS_2S_2^*=c-d=S_1S_1^*-S_2S_2^*=F
$$
We still have to impose the last condition, namely that $ZVZ^*=V^*$. We have
\begin{align*}
ZVZ^*&=\left(\begin{array}{cc}
1 & 0\\
0 & -1\\
\end{array}\right)
\left(\begin{array}{cc}
a & b\\
c & d\\
\end{array}\right)
\left(\begin{array}{cc}
1 & 0\\
0 & -1\\
\end{array}\right)
=
\left(\begin{array}{cc}
a & -b\\
-c & d\\
\end{array}\right)\\
&=\left(\begin{array}{cc}
a^* & c^*\\
b^* & d^*\\
\end{array}\right)
=\left(\begin{array}{cc}
a & b\\
c & d\\
\end{array}\right)^*
\end{align*}
which yields an additional set of equations
\begin{align}
&a^*=a\\
&b^*=-c\\
&c^*=-b\\
&d^*=d
\end{align}
However, this implies that 
\begin{align*}
&-b=c^*=(d+F)^*=d^*+F^*=d+F
\end{align*}
Now equations (5) and (7) imply
\begin{align*}
&aS_1+bS_1=-cS_1-dS_1
\end{align*}
while equations (6) and (8) imply
\begin{align*}
&aS_2+bS_2=-cS_2-dS_2
\end{align*}
Therefore, we have 
\begin{align}
&a+b=-c-d
\end{align}
Since $a=b+F$, we have $a=a^*=b^*+F^*=-c+F$.
But because $a=F+b$ and $a=F-c$, we must have $b=-c$.
From (13) we get $a=-d$.
Since $VV^*=1$, we get these new equations
\begin{align}
&c^2+d^2=1\\
&dc+cd=0
\end{align}
Since $c=d+F$, we get 
\begin{align}
&0=(d+F)^2+d^2-1=d^2+F^2+dF+Fd-1=d^2+dF+Fd\\
&0=d(d+F)+(d+F)d=d^2=0
\end{align}
So far we have found that $d^2=0$ and $c^2=1-d^2=1$.
But $d=d^*$, so $d^*d=0$ and $d=0$, $c^2=1$.
We also have $a=-d=0$ and $c=d+F=F$ and $b=-c=-F$. Therefore we get 
$$
V=\left(\begin{array}{cc}
0 & -F\\
F & 0\\
\end{array}\right)\in M_2(\IC)\otimes \CO_2
$$
from which we see that $V^{2}=-I$. In particular, the spectrum 
of $V$ is a finite set.
Thus a contradiction have been arrived at  because if such a $V$ existed, then
its spectrum should in fact be the whole one-dimensional torus: see  {\it e.g.} \cite{ACR}, where $C^*(U)\subset\CQ_2$
is shown to be $^*$-isomorphic with $C(\IT)$.

\section*{Acknowledgements}
V.A. acknowledges the support from the Swiss National Science foundation through the SNF project no. 178756 (Fibred links, L-space covers and algorithmic knot theory).

\end{document}